\documentclass[preprint,12pt,3p]{elsarticle}

\usepackage{amssymb}
\usepackage{amsthm}

\usepackage{amsmath}
\usepackage{mathtools}
\usepackage{amsfonts}
\usepackage{mathrsfs}
\usepackage{verbatim}
\usepackage{etoolbox}
\usepackage[none]{hyphenat}

\newtheorem{thm}{Theorem}
\newtheorem{cor}[thm]{Corollary}     
\newtheorem{prop}[thm]{Proposition}
\newtheorem{conj}[thm]{Conjecture}
\newtheorem{rmk}[thm]{Remark}

\newcommand{\N}{\mathbb{N}}

\journal{arXiv}

\begin{document}

\begin{frontmatter}

\title{A note on the converse of Wolstenholme's Theorem\tnoteref{label1}}
\tnotetext[label1]{This research did not receive any specific grant from funding agencies in the public, commercial, or not-for-profit sectors.}

\author{Saud Hussein}
\address{Institute of Mathematics, Academia Sinica, 6F, Astronomy-Mathematics Building, No.1, Sec.4, Roosevelt Road, Taipei 10617, Taiwan}

\ead{saudhussein@gate.sinica.edu.tw}

\begin{abstract}
Given a prime $p$ and a positive integer $m$ satisfying a certain inequality, the converse of Wolstenholme's Theorem is shown to hold for the product $mp^k$ where $k$ is any positive integer, generalizing a result by Helou and Terjanian.
\end{abstract}

\begin{keyword}
Wolstenholme's Theorem
\end{keyword}

\end{frontmatter}


\section{Jones' Conjecture}

For $n \in \N$, let $w_n = \binom{2n-1}{n-1} = \frac{1}{2}\binom{2n}{n}$. In 1862, Wolstenholme \cite{Wolstenholme} proved the following:

\begin{thm} [Wolstenholme's Theorem] If $p \geq 5$ is prime, then \[w_p \equiv 1 \pmod{p^3}.\]
\end{thm}
\noindent
James P. Jones conjectures no other solutions exist.

\begin{conj} [Jones' Conjecture]
\[w_p \equiv 1 \pmod{p^3} \iff p \geq 5 \text{ is prime}.\]
\end{conj}
\noindent
Jones' conjecture is true for even integers, powers of primes $\leq$ $10^9$ (\cite{Trevisan}, \cite{Helou}), and based on computations for integers $\leq$ $10^9$. See the expository paper \cite{Hussein} for further background.

Recall the $p$-adic valuation of an integer is the exponent of the highest power of the prime $p$ that divides the integer.

\begin{thm} [Kummer's Theorem for binomial coefficients \cite{Kummer}] \label{kum}
Given integers $n\geq m\geq 0$ and a prime $p$, the $p$-adic valuation of $\binom{n}{m}$ is equal to the number of carries when $m$ is added to $n - m$ in base $p$.
\end{thm}
\noindent
Notice the $p$-adic valuation of $w_n$ and $2w_n$ are equal when $p \not= 2$. The integers $2w_n = \binom{2n}{n}$ are known as the \textit{central binomial coefficients}.

\begin{prop} \label{prop1}
For any odd prime $p$ and $m \in \N$ such that \[p^a < m < p^{a+1} < 2m\] for some integer $a\geq 0$, the product $n=mp^b$ satisfies $w_n \not\equiv 1 \pmod{n}$ for any $b \in \N$ and therefore Jones' Conjecture holds for $n$.
\end{prop}

\begin{proof}
Let $m = n_0 + n_1 p + n_2 p^2 + \cdots + n_k p^k$ be the $p$-adic expansion of $m$. This means for each coefficient $n_i$, $0 \leq n_i \leq p-1$ with $n_k \not= 0$. Assume there exists an integer $a\geq 0$ such that \begin{align} p^a < m < p^{a+1} < 2m. \label{ineq} \end{align}
Then \[p^a <  n_0 + n_1 p + n_2 p^2 + \cdots + n_k p^k < p^{a+1},\] so $k < a+1$. Also by the definition of a $p$-adic expansion, $k \geq a$. To see this explicitly, the formula for the sum of a geometric series gives us \begin{align*} n_0 + n_1 p + n_2 p^2 + \cdots + n_k p^k &\leq (p-1)(1+p+p^2+ \cdots + p^k)\\
&= (p-1)\left(\frac{p^{k+1}-1}{p-1}\right)\\
&= p^{k+1} - 1.
\end{align*}
Therefore \[p^a < p^{k+1} - 1 \implies a \leq k < a+1 \implies k=a.\]
Now let \[m_0 + m_1 p + m_2 p^2 + \cdots + m_l p^l\] be the $p$-adic expansion of $2m$. By the same argument as before, \[m_0 + m_1 p + m_2 p^2 + \cdots + m_l p^l \leq p^{l+1} - 1,\] so by \eqref{ineq}, \[p^{a+1} < p^{l+1} - 1 \implies l \geq a + 1.\]
Therefore $l > k$, meaning there is at least one carry when adding $m$ and $2m - m = m$ in base $p$. Since $w_m = \frac{1}{2}\binom{2m}{m}$ and $p \not= 2$, theorem \ref{kum} implies $p\mid w_m$. Finally, the coefficients in the $p$-adic expansion of $m$ and $n=mp^b$, $b \in \N$, are the same, so $p\mid w_n$ and so $w_n \not\equiv 1 \pmod{p}$. Therefore $w_n \not\equiv 1 \pmod{n}$ and the proof is complete.
\end{proof}

\begin{cor}[Proposition 5, part 4 \cite{Helou}] \label{coroll}
For any odd prime $p$ and $m \in \N$ such that \[m < p < 2m,\] the product $n=mp^k$ satisfies $w_n \not\equiv 1 \pmod{n}$ for any $k \in \N$ and therefore Jones' conjecture holds for $n$.
\end{cor}

\begin{rmk} \label{rmk1}
If $p$ and $q$ are consecutive primes with $ 2 \not= q < p$, then Chebyshev's Theorem \cite{Chebychev} (also known as Bertrand's postulate) implies \[q < p < 2q.\]
Therefore by corollary \ref{coroll}, Jones' conjecture holds for the product of two consecutive primes. 
\end{rmk}

\begin{rmk}
For powers of a prime, proposition \ref{prop1} clearly does not apply. In this case since $n = p^k$, adding $n$ and $n$ in base $p$ gives us $2p^k$ and so there are no carries when $p \not= 2$. Therefore by theorem \ref{kum}, $p\not \vert \,\, w_n$.
\end{rmk}


\bibliographystyle{elsarticle-harv}

\bibliography{biblio.bib}

\end{document}